\documentclass[a4paper,11pt]{amsart}
\usepackage[T1]{fontenc}
\usepackage{amsthm,amsmath,amssymb}
\usepackage{mathrsfs}
\usepackage{enumerate}
\usepackage{graphicx}
\usepackage{ mathdots }

\theoremstyle{plain}

\newtheorem{theorem}{Theorem}[section]
\newtheorem{thmIntro}{Theorem}[]
\newtheorem{lemma}[theorem]{Lemma}
\newtheorem{proposition}[theorem]{Proposition}
\newtheorem{fact}[theorem]{Fact}

\newtheorem{question}[theorem]{Question}
\newtheorem{corollary}[theorem]{Corollary}
\theoremstyle{definition}
\newtheorem{definition}[theorem]{Definition}
\newtheorem{ex}[theorem]{Example}

\newtheorem{remark}[theorem]{Remark}

\newcommand{\Z}{\mathbb{Z}}
\newcommand{\N}{\mathbb{N}}
\newcommand{\tp}{\mathrm{tp}}
\newcommand{\abs}[1]{\ensuremath{\left|#1\right|}}

\title{On Superstable Expansions of Free Abelian Groups}
\author{Daniel Palac\'in and Rizos Sklinos}

\address{Institut f\"ur Mathematische Logik und Grundlagenforschung,
    Universit\"at M\"unster;  Einsteinstrasse 62,  M\"unster 48149, Germany}
\email{daniel.palacin@uni-muenster.de}

\address{Institut Camille Jordan UMR5208, Universit\'e de Lyon 1;  43 bd du 11
novembre 1918, Villeurbanne 69622, France}
\email{sklinos@math.univ-lyon1.fr}

\thanks{The first author was supported by the project Groups, Geometry \&
Actions (SFB 878) and the project MTM2014-59178-P.
The second named author was supported by the LABEX MILYON (ANR-10-LABX-0070) of
Universit\'e de Lyon, within the program "Investissements d'Avenir"
(ANR-11-IDEX-0007) operated
by
the French National Research Agency (ANR).
We would like to thank Bruno Poizat for providing a reference of his result, as well as
Anand Pillay  and Uri
Andrews for their valuable comments. Finally, we are indebted to Gabe Conant for pointing out some mistakes towards the proof of Theorem 2 in an earlier version of this article.}

\begin{document}

\begin{abstract}
We prove that $(\Z,+,0)$ has no proper superstable expansions of finite Lascar
rank. Nevertheless, this structure equipped with a predicate defining powers of
a given natural number is superstable of Lascar rank $\omega$. Additionally,
our methods yield other superstable expansions such as $(\Z,+,0)$ equipped with
the set of factorial elements.
\end{abstract}

\maketitle


\section{Introduction}\label{intro}
This paper fits into the general framework of adding a new predicate to a well
behaved structure and
asking whether the obtained structure is still well behaved. A similar line of
thought is to impose the desired
properties on the expanded structure and ask for which predicates these
properties are fulfilled. Even more, one might
ask whether there exist proper expansions fulfilling the desired properties.

Many results that belong to the above mentioned framework have been obtained by
various authors. For example
Pillay and Steinhorn proved that there are no (proper) o-minimal expansions of
$(\N,\leq)$. On the other hand,
Marker \cite{MarkerExp} proved that there are (proper) strongly minimal
expansions of $(\N,s)$, i.e. the natural numbers with the successor function.
In a more abstract context Baldwin and Benedikt proved that if $\mathcal{M}$ is
a stable structure and $I$
is a {\em small} set of indiscernibles then $(\mathcal{M},I)$ is still stable.
Finally, Chernikov and Simon \cite{SimCheExt} proved the analogous
result for NIP theories, i.e. NIP is preserved after naming a {\em small}
indiscernible sequence.

In this short paper we are mainly interested in (finitely generated) free
abelian groups. We are motivated by
the recent addition of torsion-free hyperbolic groups to the family of
stable groups (see \cite{SelaStability}). In a torsion-free hyperbolic group
centralizers of
(non-trivial) elements are infinite cyclic and one is interested in the induced
structure on them.
It seems that understanding the induced structure on these centralizers boils
down to understanding whether they
are superstable and if so calculate their Lascar rank.

Our main result generalizes
a theorem in the thesis of the second named author proving that every Lascar
rank $1$
expansion of $(\Z,+,0)$ is a pure group (see \cite[Theorem
8.2.3]{ThesisSklinos}).

\begin{thmIntro}\label{MainTheorem}
There are no (proper) superstable finite Lascar rank expansions of $(\Z,+,0)$.
\end{thmIntro}

We also show that one cannot strengthen the above result any further by proving:

\begin{thmIntro}\label{MainTheorem2}
The theory of $(\Z,+,0,\Pi_q)$ is superstable of Lascar rank $\omega$, where
$\Pi_q$ denotes the set of powers of a natural number $q$.
\end{thmIntro}

In fact, our methods can be used to provide other superstable expansions
by adding other sets such as sets of the form $\{q^{p^n}\}_{n<\omega}$ for
some natural
numbers $p,q$ or the set of factorial elements, see Proposition
\ref{Fac}.
On the other hand, if one moves to higher rank free abelian groups Theorem
\ref{MainTheorem} is no longer true, and it is not hard to
find proper superstable Lascar rank $1$ expansions of $(\Z^n,+,0)$, for $n\geq
2$. The main reason being that there exist finite index subgroups
of $\Z^n$ (for $n\geq 2$) that are not definable in $(\Z^n,+,0)$. Still, we
record that a
superstable finite Lascar rank expansion of $(\Z^n,+,0)$ is one-based and has
Lascar rank at most $n$.

While checking our results, the second author figured out in a talk of Bruno
Poizat that Theorem \ref{MainTheorem2} was already proved in \cite[Th\'eor\`eme
25]{PoizatSuperGen}. Nevertheless, as both approaches are completely distinct we
believe that it is worth recording our result since, as we have
already pointed out, it yields distinct examples. Moreover, to our
knowledge, Theorem \ref{MainTheorem} was unknown. The essential tools to prove
it come from geometric stability. We combine results from Hrushovski's thesis
together with Buechler's dichotomy theorem, the characterization of one-based
groups by Hrushovski-Pillay and a result on one-based types due to Wagner.

\section{Finite rank expansions}

The aim of this section is to study superstable expansions of finite Lascar
rank of the structure $(\Z^n,+,0)$. We assume the reader is familiarized with
the general theory of geometric stability, see
\cite{PillayStability,WagnerStableGroups} as a reference. In addition we
require the following result which characterizes
subgroups of finitely generated free abelian groups.

\begin{fact}\label{FreeAbelSubs}
Let $G$ be a subgroup of $\Z^n$. Then there is a basis $(z_1,\ldots,z_n)$ of
$\Z^n$ and a sequence of
natural numbers $d_1,\ldots,d_k$ (with $d_i$ dividing $d_{i+1}$ for $i<k$), such
that $(d_1z_1,\ldots,d_kz_k)$
forms a basis of $G$.
\end{fact}

One can use Fact \ref{FreeAbelSubs} to prove the following lemma, which we
consider as being part of the folklore.

\begin{lemma}
Let $G$ be a subgroup of $\Z^n$. Then $G$ is definable in $(\Z,+,0)$.
\end{lemma}

Now, we prove Theorem \ref{MainTheorem}.

\begin{proof}[Proof of Theorem 1]
Consider a finite Lascar rank expansion $\mathcal Z=(\Z,+,0,\ldots)$
of
$(\Z,+,0)$, and let $\Gamma\succeq\mathcal Z$ be an enough saturated
elementary extension. As $\Gamma$ has finite Lascar rank, its principal generic
type is non-orthogonal to a type $q$ of Lascar rank one and hence, we can find
an $\emptyset$-definable normal subgroup $H$ of infinite index in $\Gamma$ in a
way that $\Gamma/H$ is $\mathcal Q$-internal, where $\mathcal Q$ is the
family of all $\emptyset$-conjugates of $q$. In fact, since $H$ is
defined without parameters, the subgroup $H\cap\Z$ has infinite index
in $\Z$, hence $H\cap \Z$ must be trivial, and so is $H$. This yields
that $\Gamma$ is
$\mathcal Q$-internal. On the other hand, as $\Gamma$ is not
$\omega$-stable, by Buechler's dichotomy theorem $q$ must be a one-based
type and so are all its conjugates. Thus $\Gamma$ is
one-based by \cite[Corollary 12]{WagnerOnebased}, and so is the theory of
$\mathcal Z$. Thus, by the characterization of
one-based stable groups \cite[Corollary 4.4.8]{PillayStability}, every
definable subset of $\Z^n$ in the expanded structure is a boolean combination of
cosets of definable subgroups of $\Z^n$ and therefore, any definable set in
the theory of $\mathcal Z$ is already definable in the theory of $(\Z,+,0)$ by
the previous lemma, as
desired.
\end{proof}

We note, in contrast, that not all subgroups of $\Z^n$ are definable in
$(\Z^n,+,0)$. For example, the finite index subgroup $3\Z\oplus 2\Z$ of
$\Z^2$ is not definable in $(\Z^2,+,0)$, and of course any non-trivial infinite
index subgroup of $\Z^n$, for $n\geq 2$, is not definable in $(\Z^n,+,0)$.

\begin{theorem}
Any finite Lascar rank expansion of $(\Z^n,+,0)$ is one-based and has Lascar
rank at most $n$.
\end{theorem}
\begin{proof}
Consider a finite Lascar rank expansion
$\mathcal Z=(\Z^n,+,0,\ldots)$ of $(\Z^n,+,0)$. A similar argument as in the
previous theorem yields that the theory of $\mathcal Z$ is one-based. For this,
let $\Gamma\succeq\mathcal Z$ be an enough saturated model. As it has finite
Lascar rank by assumption, the general theory yields the existence of a finite
series of $\emptyset$-definable normal subgroups
$$\Gamma= H_0\unrhd H_1 \unrhd \ldots \unrhd H_{m+1}\unrhd \{0\}$$
such that $H_{m+1}$ is finite and each factor $H_i/H_{i+1}$ is infinite and internal
to a family $\mathcal Q_i$ of $\emptyset$-conjugates of some type $q_i$ of
Lascar rank one. Since free abelian groups are torsion-free they do not have any
finite (non-trivial) subgroups, and so neither does $\Gamma$. This implies that
$H_{m+1}$ is trivial. Furthermore, by Fact \ref{FreeAbelSubs} we obtain
that no infinite quotient of $\Z^n$ is $\omega$-stable. As all subgroups
$H_i$ are $\emptyset$-definable, we deduce that the quotients $H_i/H_{i+1}$
cannot have ordinal Morley rank, and neither do the types from the families
$\mathcal Q_i$. Whence, we conclude by Buechler's
dichotomy theorem that all of them are one-based, and so is $\Gamma$ again by
\cite[Corollary 12]{WagnerOnebased}.

To see that the expansion $\mathcal Z$ has Lascar rank at most $n$, consider
the structure $\mathcal Z_{\mathrm{proj}}$ given as $(\Z^n,+,0,P_1,\ldots,P_n)$,
where the predicate $P_i$ is interpreted as the projection of $\Z^n$ onto its
$i$th coordinate. It is clear that $\mathcal Z_{\mathrm{proj}}$ is interpretable
in $(\Z,+,0)$ and so it has Lascar rank $n$. On the other hand, since $\mathcal
Z$ is one-based, it is interpretable in $\mathcal Z_{\mathrm{proj}}$ by the
characterization of one-based stable groups \cite[Corollary
4.4.8]{PillayStability} and thus, it has Lascar rank at most $n$.
\end{proof}

\begin{remark}
Observe that the proof yields that any superstable finite Lascar rank expansion
of $(\Z^n,+,0)$ is interpretable in the structure $\mathcal Z_{\mathrm{proj}}$.
\end{remark}

\section{Superstable expansions of $(\Z,+,0)$}

In this section we shall see that there are proper superstable expansions of
$(\Z,+,0)$, necessarily, by Theorem \ref{MainTheorem}, of infinite Lascar rank.

\begin{definition}
Let $\mathcal{L}$ be a first-order language and $P(x)$ a unary predicate. We
denote by $\mathcal{L}_P$ the first-order language $\mathcal{L}\cup\{P\}$.
We say that an $\mathcal{L}_{P}$-formula $\phi(\bar{y})$ is bounded (with
respect to $P$) if it has the form
$$
Q_1x_1\in P\ldots Q_nx_n\in
P \, \psi(\bar{x},\bar{y}),
$$
where the $Q_i$'s are quantifiers and $\psi(\bar{x},\bar{y})$ is an
$\mathcal{L}$-formula.
\end{definition}

The following theorem will be useful for proving Theorem \ref{MainTheorem2}, we
refer the reader to \cite{CasaZie} for the proof.

\begin{theorem}\label{FactCasZieg}
Let $\mathcal{M}$ be an $\mathcal{L}$-structure and $A\subseteq M$. Consider
$(\mathcal{M},A)$ as a structure in the
expanded language $\mathcal{L}_{P}:=\mathcal{L}\cup\{P\}$.
Suppose every $\mathcal{L}_{P}$-formula in $(\mathcal{M},A)$ is equivalent to a
bounded one. Then, for every $\lambda\geq\abs{\mathcal{L}}$, if
both $\mathcal{M}$ and $A_{ind}$ are $\lambda$-stable, then $(\mathcal{M},A)$
is $\lambda$-stable.
\end{theorem}

Let $\equiv_n$ be the congruence modulo $n$ relation on the integers. Observe
that $a\not\equiv_n b$ is equivalent to $a\equiv_n b+1\lor a\equiv_n
b+2\lor\ldots \lor a\equiv_n b+(n-1)$, and hence we get the following
remark.

\begin{remark}\label{NormalDisjunctiveForm}
Let $\mathcal{L}_{mod}$ be the language of groups expanded with countably many
$2$-place predicates.
We recall that an $\mathcal{L}_{mod}$-formula $\phi(\bar{x})$ is equivalent, in
$(\mathbb{Z},+,0,\{\equiv_n\}_{n<\omega})$,
to a finite disjunction of formulas of the form:
\begin{center}
\begin{tabular}{lllll}
$t_1(\bar{x})=0$ & $\land $ & $\ldots$ & $\land$ & $t_k(\bar{x})=0$ \\
$r_1(\bar{x})\neq 0$ & $\land $ & $\ldots$ & $\land$ & $r_l(\bar{x})\neq 0$ \\
$s_1(\bar{x})\equiv_{n_1}0$ & $\land $ & $\ldots$ & $\land$ &
$s_m(\bar{x})\equiv_{n_m}0$ \\
\end{tabular}
\end{center}
where $t_i(\bar{x}), s_i(\bar{x}), r_i(\bar{x})$ are terms in the above
language.
\end{remark}



Set $\Pi_q$ to denote the set $\{q^n \ | \ 1\leq n<\omega\}$ for some natural number $q$.

\begin{lemma}\label{Consistency}
Let $q$ be a natural number. Let $\bar{b}$ be a tuple in $\Z$ and
$\phi(\bar{x},y,\bar z)$ be an $\mathcal{L}$-formula, where $\mathcal{L}$ is the
language of groups.
Suppose that the set $\Gamma(y):=\{\phi(\bar{b},y,\bar\alpha) \ | \ \bar\alpha\in \Pi_q^{|\bar z|}\}$ is
consistent with $\mathcal{T}h(\Z,+,0)$. Then there exists
$c\in\Z$ realizing the set $\Gamma(y)$.
\end{lemma}
\begin{proof}
We may assume that $\phi(\bar{x},y,\bar\alpha)$ is a formula as in Remark
\ref{NormalDisjunctiveForm}. If we fix some tuple $\bar \alpha_0$ in
$\Pi_q$, then each disjunctive clause in $\phi(\bar{b},y,\bar \alpha_0)$ asserts that
$y$ is equal to some element from a finite list of elements in $\Z$,
and $y$ is not equal to any element from a finite list of elements in $\Z$ and
$y$ belongs to the intersection of finitely many cosets of
fixed subgroups of $\Z$, where these fixed subgroups only depend on $\phi$ (not
$\bar{b}$ or $\bar\alpha_0$).

Our assumption that $\Gamma(y)$ is consistent implies that for each tuple $\bar\alpha_0$
in $\Pi_q$ we may choose a disjunctive
clause in $\phi(\bar{b},y,\bar\alpha_0)$ such that the set of these clauses is again
consistent. Note that if one of the chosen
clauses involves an equality, then the result holds trivially. So we will assume
that no equality is involved in any
disjunctive clause of $\phi$. On the other hand the intersection of cosets of
subgroups of a group is either empty or a coset
of the intersection of the subgroups, thus we may assume that a disjunctive
clause that involves congruence modulo relations,
it involves exactly one.

Next we prove that a finite union of sets of the form
$$\{k_0+k_1\cdot\alpha_1+\ldots+ k_s\cdot\alpha_s \ | \
\alpha_1,\ldots,\alpha_s\in \Pi_q\}$$
 cannot cover any coset of any (non-trivial) subgroup
of $\Z$. Suppose otherwise that the coset $m+n\Z$ is contained in a such finite union, and observe that we may assume, after subtracting $m$ if necessary, that $m=0$.
Thus, for each set of the above form we can write each given coefficient $k_i$ in base $q$ and obtain a natural number $l$ such that $n\Z$ is covered by finitely many sets of the form
$$
\left\{\lambda_0+\lambda_1\cdot\alpha_1+\ldots+\lambda_l\cdot\alpha_l \ | \ \alpha_1,\ldots,\alpha_l\in\Pi_q, \ 0\le |\lambda_0|,\ldots,|\lambda_l|<q \right\}.
$$
Assume $l$ is the biggest number obtained in the above mentioned fashion. Then, any multiple of $n$ can be written in base $q$ with at most $l+1$ many summands. Now, let $\mu$ be the element
$n\cdot(1+q+q^2+\ldots+q^{l+1})$, which clearly belongs to $n\Z$.
After writing $n$ in base $q$, we obtain that $\mu$ is written in base $q$ as the sum of at least $l+2$ many summands. Thus, by the uniqueness of the representation of $\mu$ in base $q$,
we obtain a contradiction.

Now, the consistency of $\Gamma(y)$ implies that $y$ belongs to the intersection
of finitely many cosets of subgroups of $\Z$ and $y$ is not equal
to any element of a finite union of sets of the form $$\{k_0+k_1\cdot\alpha_1+\ldots+ k_s\cdot\alpha_s \ | \
\alpha_1,\ldots,\alpha_s\in \Pi_q\}.$$
By the previous paragraph, a solution can be found in $\Z$ and this finishes the
proof.
\end{proof}


Now we are able to prove the following technical lemma.

\begin{lemma}\label{ReplaceSmallness}
Let $q$ be a natural number. Let $\mathcal{L}$ be the language of groups and
$P(x)$ be a unary predicate.
Let $\mathcal{Z}:=(\Z,+,0,\Pi_q)$ be an $\mathcal{L}_P$-structure.

Let $\phi(\bar{x},y,\bar z)$ be an $\mathcal{L}$-formula. Then there exists
$k<\omega$ such that:
$$
\mathcal{Z}\models\forall\bar{x}\big((\forall\bar z_0\in
P\ldots\forall\bar z_k\in P \, \exists y
\bigwedge_{j\le k}\phi(\bar{x},y,\bar z_j))
\rightarrow\exists y\forall \bar z\in P \phi(\bar{x},y,\bar z)\big).$$
\end{lemma}
\begin{proof}
Since $(\Z,+,0)$ has nfcp we can assign to each formula $\phi$ a natural number
$k$ such that
any set of instances of the formula $\phi$ is consistent if and only if it is
$k$-consistent. By
Lemma \ref{Consistency} if a set
$\{\phi(\bar{b},y,\bar\alpha) \ | \ \bar\alpha\in
\Pi_q^{|\bar z|}\}$ is consistent, then
a solution can be found in $\Z$ and this is enough to conclude. \end{proof}

The following proposition is an easy corollary of Lemma \ref{ReplaceSmallness}
and the proof is left to the reader, see \cite[Proposition 2.1]{CasaZie}.

\begin{proposition}\label{Bounded}
Let $q$ be a natural number. Let $\mathcal{L}$ be the language of groups and
$P(x)$ be a unary predicate. Let $\mathcal{Z}:=(\Z,+,0,\Pi_q)$ be an
$\mathcal{L}_P$-structure.
Then every $\mathcal{L}_P$-formula in $\mathcal{Z}$ is bounded.
\end{proposition}

As a consequence we deduce:

\begin{corollary}\label{CorTypes}
Let $q$ be a natural number. Let $\mathcal{L}$ be the language of groups and
$P(x)$ be a unary predicate, and let  $(\Gamma,+',0,\Pi_q')\equiv(\Z,+,0,\Pi_q)$
be
$\mathcal{L}_P$-structures.
Two tuples of $\Gamma$ realize the same $\mathcal L_P$-formulas over any set of
parameters $C\subseteq \Gamma$
whenever they realize the same $\mathcal L$-formulas over $\Pi_q'\cup C$.
\end{corollary}
\begin{proof}
Let $a$ and $b$ be two tuples realizing the same $\mathcal L$-formulas over
$\Pi_q',C$. It is easy to see by induction on the number
of quantifiers that $a$ and $b$ realize the same formulas of the form
$$Q_1x_1\in P\ldots Q_nx_n\in P \, \psi(\bar{x},\bar{y}),$$
where the $Q_i$'s are quantifiers and $\psi(\bar{x},\bar{y})$ is an
$\mathcal{L}(\Pi_q'\cup C)$-formula. Hence, we conclude by Proposition
\ref{Bounded}.
\end{proof}

Our last task is to prove that the induced structure on the subset of the
integers that consists of powers of some natural number, coming from $(\Z,+,0)$,
is tame.
Recall that if $B$ is a subset of the domain $M$, of a first order structure
$\mathcal{M}$,
then by the {\em induced structure on $B$} we mean the structure with domain $B$
and predicates for every subset of $B^n$ of the
form $B^n\cap \phi(M^n)$, where $\phi(x)$ is a first-order formula (over the
empty set). We denote this structure by $B^{\rm ind}$.

\begin{proposition}\label{PropositionInduced}
Let $q$ be a natural number. The structure $\Pi_q^{\rm ind}$ (with respect to
$(\Z,+,0)$) is superstable and has Lascar rank one.
\end{proposition}

The proof is split in a series of lemmata. We first prove some results, we
believe well known, in the spirit of Diophantine analysis.

\begin{lemma}\label{LemmaArith}
Let $q$ be some natural number. Let $k<n$ be natural numbers such that $n$ is
co-prime with $q$,
and let $[k]_n$ denote the congruence class of $k$ modulo $n$.
Then $\Pi_q\cap [k]_n=\{q^{m_0+\varphi(n)\cdot m}: m<\omega\}$, where
$\varphi(n)$ is the Euler's phi function and
$m_0$ is the smallest natural number for which $q^{m_0}\equiv k\mod n$.
\end{lemma}
\begin{proof}
We first note that if $k,n$ are not co-prime then the intersection of $[k]_n$
with $\Pi_q$ is empty. The common factor
of $k$ and $n$ does not contain a factor of $q$ since $n$ is co-prime with $q$,
and it should appear as factor in any element of $k+n\cdot\Z$.

We now assume that $k,n$ are co-prime and we fix $k,n,m_0$ satisfying the
hypothesis of the lemma. We define $\lambda_m$ recursively as follows:
$$
\begin{array}{l}
\lambda_0:=\frac{q^{m_0}-k}{n} \\
\lambda_{m+1}:=\lambda_m\cdot b^{\varphi(n)}+k\cdot\frac{q^{\varphi(n)}-1}{n}, \
\textrm{for} \  0\leq m<\omega.
\end{array}
$$
Note that, by Euler's theorem, all the $\lambda_m$'s are integers.
Furthermore, one can easily see, by induction on $m$,  that $\lambda_m\cdot n+k$
is a power of $q$ of the form
$q^{m_0+\varphi(n)\cdot m}$ and therefore $\{\lambda_m\cdot n+k \ | \
m<\omega\}\subseteq \Pi_q\cap [k]_n$.

In fact, the other inclusion also holds. To see this, let $q^l$ be an arbitrary
power of $q$.
We may assume that $l>m_0$, since $m_0$ is the smallest natural number
satisfying the hypothesis.
Then we can find some $m$ such that $l=m_0+\varphi(n)\cdot m+s$ with
$s<\varphi(n)$.
As $\varphi(n)$ is the order of the multiplicative group $(\Z/n\Z)^\times$, we
get $q^s\in [1]_n$ only when $s=0$. Since $k,n$ are co-prime
$k$ has a multiplicative inverse modulo $n$. Therefore
$$q^l=q^{m_0+\varphi(n)\cdot m}\cdot q^s\equiv_n k\cdot q^s\equiv_n k
\ \mbox{ if and only if } \ s=0,$$ and this
concludes the proof.
\end{proof}

\begin{remark}\label{Power2}
Let $q$ be some natural number. Assume $n$ is a power of a prime which is not
co-prime with $q$, then
the intersection of $\Pi_q$ with $[k]_n$ is either finite or co-finite in
$\Pi_q$.
\end{remark}


\begin{lemma}\label{Equations}
Let $k_1x_1+\ldots+k_nx_n=k$ be an equation over the integers and $S\subseteq
\Z^n$ be its solution set. Then $S\cap \Pi_q^n$ is either empty or a finite union
of sets of the form: \\ \\
\begin{tabular}{cl}  $\{ (q^{\lambda_1},\ldots, q^{\lambda_n}) \ |$ & $\lambda_{i_1}>m_1, \ldots, \lambda_{i_k}>m_k,$ \\ & $\lambda_{i_{k+1}}=\alpha_{k+1}\lambda_{i_{j_1}}+ m_{k+1},$
\\ & $\vdots$ \\ & $\lambda_{i_n}=\alpha_n\lambda_{i_{j_{n-k+1}}}+m_n\},$ \end{tabular} \\ \\
 where $m_1,\ldots, m_n\in \Z$, $\alpha_i\in\{0,1\},$  and $\{i_{j_1},\ldots, i_{j_{n-k+1}}\}\subseteq \{i_1,\ldots,i_k\}$.
\end{lemma}
\begin{proof}
The proof is by induction. For the base case $n=1$, we easily see that $k_1x_1=k$ can either be empty or have a single solution,
thus the solution set is of the required form. Suppose that for every $m<n$ the solution set of any linear equation in $m$ variables
have the required form, we show that the same holds for equations with $n$ variables.

We split the solution set in finitely many subsets according to the finitely many orderings we can put on the $n$ variables.
For example to the ordering $x_1\leq x_2\leq\ldots\leq x_n$ corresponds the subset of solutions for which each co-ordinate
takes bigger or equal value to its previous one. We analyse those subsets in parallel. For notational purposes
we analyse the set with the above ordering. Let $\{(q^{\lambda_1(i)},\ldots, q^{\lambda_n(i)}) \ | \ i<\omega\}$ be an
enumeration of this set. Then
$$q^{\lambda_1(i)}(k_1 + k_2q^{\lambda_2(i)-\lambda_1(i)}+\ldots+k_nq^{\lambda_n(i)-\lambda_1(i)})=k$$
We take cases:\\
{\bf Case 1)} Suppose the sequence $\lambda_1(i)$ is bounded. Then for each of the finitely many values of $\lambda_1(i)$
we have $k_2q^{\lambda_2(i)-\lambda_1(i)}+\ldots+k_nq^{\lambda_n(i)-\lambda_1(i)} = \frac{k}{q^\lambda_1(i)} - k_1$. Using the
inductive hypothesis for the linear equation $k_2x_2+\ldots +k_nx_n=\frac{k}{q^\lambda_1(i)} - k_1$, we see that
the solution set is contained in a set of the required form. \\
{\bf Case 2)} Suppose the sequence $\lambda_1(i)$ is unbounded. Then $k$ must be $0$ and
$k_1 + k_2q^{\lambda_2(i)-\lambda_1(i)}+\ldots+k_nq^{\lambda_n(i)-\lambda_1(i)}=0$. Thus, we have:
$$q^{\lambda_2(i)-\lambda_1(i)}(k_2+\ldots+k_nq^{\lambda_n(i)-\lambda_2(i)}) = -k_1$$
Note that in this case, since $k_1\neq 0$ we must have that $\lambda_2(i)-\lambda_1(i)$ is bounded.
For each of the finitely many values $\lambda_2(i)-\lambda_1(i)$ takes, we continue our analysis in parallel. We have:
$$k_3q^{\lambda_3(i)-\lambda_2(i)}+\ldots+k_nq^{\lambda_n(i)-\lambda_2(i)} = \frac{-k_1}{q^{\lambda_2(i)-\lambda_1(i)}}-k_2$$
At this step and every step after we take cases according to whether $\lambda_{j+1}(i)-\lambda_j(i)$ is bounded or not.
In the case where it is bounded, for each value of the finitely many, a relation of the form $\lambda_{j+1}=\lambda_j+m_j$ is introduced. In the case it is
unbounded we use the induction hypothesis as our solution set is contained in the solution set of linear equations of the form
$k_1x_1+\ldots+k_mx_m=0$ and $k_{m+1}x_{m+1}+\ldots +k_nx_n=0$.
\end{proof}

\begin{lemma}\label{LemmaInd1}
Let $q$ be some natural number. Let
$\mathcal{N}:=(\N,s,\{Q_{k,n}\}_{n<\omega,k<n})$ be a first order structure
where the function symbol $s$ is interpreted as
the successor function and the predicate $Q_{k,n}$ is interpreted as
the set of natural numbers which are residual to $k$ modulo $n$. Then
$\Pi_q^{\rm ind}$ is definably interpreted in $\mathcal{N}$.
\end{lemma}

\begin{proof}
Throughout the proof the symbol $s^m$ will be used to denote $s\circ
s\circ\ldots\circ s$ $m$-times.
We also allow $m$ to be negative, in which case $s^m$ denotes the composition of
the predecessor function $m$-times (which is clearly definable).

We first interpret $\Pi_q$ to be the domain of $\mathcal{N}$.
Now let $P$ be a predicate of $\Pi_q^{\rm ind}$. By the construction of
$\Pi_q^{\rm ind}$ we have that $P$ is a subset of the form $\phi(\Z^n)\cap
\Pi_q^n$ for some
quantifier free formula $\phi$ in $(\Z,+,0, \{\equiv_n\}_{n<\omega})$. Since a
quantifier free formula is a boolean
combination of formulas of the form $t(\bar{x})=0$ and $s(\bar{x})\equiv_l0$, we
only need to
interpret in $\mathcal{N}$ solution sets of equations and congruence relations
of the above simple form intersected with $\Pi_q^n$.

Suppose $\phi(\bar{x})$ is the equation $t(\bar{x})=0$. Then, by Lemma
\ref{Equations}, the set $\phi(\Z^n)\cap \Pi_q^n$ can be interpreted as a
finite union of sets, that for the sake of clarity can be assumed to have the following form:
$$\bigwedge_{1\le i<n} x_1=s^{m_i}(x_{i+1}) \wedge \bigwedge_{1\le j\le k}
x_1\neq j.$$
Otherwise, suppose $\phi(\bar{x})$ is the congruence relation
$s(\bar{x})\equiv_l0$. If $(r_1,\ldots,r_n)$ is a tuple of integers that satisfy
the congruence relation,
then any tuple $(q_1,\ldots,q_n)$ for $q_i\in [r_i]_l$  satisfies this relation.
Note that we can only have finitely many solutions up to $l$-congruence.
Moreover, we may assume, by the Chinese remainder theorem, that $l$ is a power
of a prime number. Thus, by Lemma \ref{LemmaArith} and Remark \ref{Power2},
$\phi(\Z^n)\cap \Pi_q^n$ can be interpreted as a finite union of sets of the
form
$$\bigwedge_{1\le i \le n} Q_{k_i,m_i}(x_i) \land \mbox{"$x_i$ is not equal to
finitely
many elements"}.$$
This finishes the proof.
\end{proof}

\begin{lemma}\label{LemmaInd2}
The theory of $\mathcal{N}:=(\N,s,\{Q_{k,n}\}_{n<\omega,k<n})$ admits quantifier
elimination after adding a constant and a unary function symbol.
Moreover it is superstable and has Lascar rank one.
\end{lemma}
\begin{proof}
We add a constant to name $1$ and a function symbol $s^{-1}$ to name the
predecessor function;\
observe that both are definable in $\mathcal{N}$.

We prove elimination of quantifiers by induction on the complexity of the
formula $\phi$. It is enough to consider the case where $\phi(\bar{x})$ is a
consistent formula of the form
$\exists y\, \psi(\bar x,y)$, where $|y|=1$ and $\psi(\bar x,y)$ is a quantifier
free formula. We can clearly assume that $\psi$ is in normal disjunctive form.
Thus, since the negation of $Q_{k,n}$ is equivalent to
the conjunction $\bigvee_{l\neq k} Q_{l,n}$, it is enough to consider the case
where $\psi(\bar x,y)$ is a finite conjunction of formulas of the following
form:
$$Q_{k,n}(x_i) \ \land \ Q_{l,m}(y) \ \land \ x_i=c \ \land \ y=d \ \land
x_i\neq a \ \land \ y\neq b $$
$$ \land \ s^{p}(x_i)=x_j \ \land s^{r}(x_l)=y \land s^{f}(x_i)\neq x_j \land
s^{g}(x_l)\neq y$$

Furthermore, we split $\psi$ to a conjunction $\psi_0(\bar{x},y)\land
\psi_1(\bar{x})$, where $\psi_1$ is the conjunction of the atomic formulas of
$\psi$ that do not contain $y$.
Clearly we may assume that $\psi_0(\bar{x},y)$ does not contain instances of
the form $y=d$ or $s^{g}(x_i)= y$. We claim that $\exists y\psi_0(\bar{x},y)$
is equivalent to $\bar{x}=\bar{x}$.
Indeed, the projection of any formula of the form
$$Q_{k,n}(y)\land\bigwedge_{1\le i\le k} s^{g^i}(x)\neq y \land \bigwedge_{1\le
j\le l} y\neq d_j$$
is equivalent to $x=x$, thus the claim follows and $\psi(\bar{x},y)$ is
equivalent to $\psi_1(\bar{x})$. So, we obtain the first part of our statement.

Quantifier elimination allows us to prove by an easy counting types argument
that the theory is superstable. Fix a set of parameters $B$. Clearly any
non-algebraic type over $B$ extends the set $\pi(x)$ given by $\{s^n(x)\neq a:
a\in B, n\in \Z\}$. Hence, by the elimination of quantifiers,
we obtain that any complete non-algebraic type over $B$ (in one variable) is
equivalent to $\pi(x)\cup\pi_0(x)$, where $\pi_0(x)$ is a complete type without
parameters. Whence, $|S(B)|=|B| + |S(\emptyset)|$, as desired. In fact, any type
without parameters is determined by positive formulas since, as noted before,
the formula $\neg Q_{k,n}(x)$
is equivalent to a disjunction of formulas $Q_{l,n}(x)$ for $l\neq k$.
In addition, as for any $n\in \N$ the formula $Q_{k,n}(x)\wedge Q_{l,n}(x)$ is
inconsistent for distinct $l,k<n$, every complete type contains only one
predicate of the form $Q_{k,n}(x)$ for a given $n$. Thus, it is easy to see that
there are continuum many types without parameters; for instance, note that the
predicate $Q_{k,2^n}(x)$ splits into $Q_{k,2^{n+1}}(x)$ and
$Q_{k+2^n,2^{n+1}}(x)$ when $k$ is odd. Hence $|S(B)|=|B|+2^\omega$ and whence,
the theory is not $\omega$-stable.

Finally, again by quantifier elimination it is easy to see that the only
formulas that divide are the algebraic ones. This shows that the theory has
Lascar rank one; the details are left to the reader.
\end{proof}

Now, the proof of Proposition \ref{PropositionInduced} follows from Lemma
\ref{LemmaInd1} and \ref{LemmaInd2}.  We can prove our second main theorem.

\begin{proof}[Proof of Theorem \ref{MainTheorem2}]
It follows from Proposition \ref{PropositionInduced} together with Theorem
\ref{FactCasZieg} that the expanded structure $(\Z,+,0,\Pi_q)$ is superstable.
As it is a proper expansion of $(\Z,+,0)$, it has infinite Lascar rank by
Theorem \ref{MainTheorem}.
Whence, it remains to see that it has Lascar rank $\omega$. For this, it is
enough to show that any forking extension of the principal generic has finite
Lascar rank.

We shall work in an enough saturated extension of $(\Z,+,0,\Pi_q)$, where
$\Pi_q$ is interpreted as $\Pi_q'$. Let $p\in S(\emptyset)$ be the
generic of the connected component, and let $q=\tp(b/B)$ be an extension of $p$.
Consider a realization $a$ of $p|B$, and note using
Lemma \ref{LemmaInd2} that $\Pi_q'$ has Lascar rank one. Now, working in the
theory of $(\Z,+,0)$,
we obtain that $\tp(b/\Pi_q',B)$ is the principal generic whenever $b\not\in{\rm
acl}(\Pi_q',B)$. Moreover, if a finite tuple $d$ is
algebraic over $\Pi_q'\cup B$ and this is exemplified by some finite tuple
$(c_1,\ldots,c_n)$ in $\Pi_q'$, then we have in $\mathcal{T}h(\Z,+,0,\Pi_q)$
that ${\rm U}(d/B)\le {\rm U}(\bar c/B)<\omega$ as the set
$\Pi_q'\times\stackrel{n}{\ldots}\times\Pi_q'$ has Lascar rank $n$.
Hence $a\not\in{\rm acl}(\Pi_q',B)$ in the sense of $(\Z,+,0)$ and hence its
type over $\Pi_q'\cup B$ is the principal generic.
Thus, by Corollary \ref{CorTypes} we deduce that $p|B=\tp(b/B)$ whenever $b$ is
not algebraic in the sense of $(\Z,+,0)$ over $\Pi_q'\cup
B$. Therefore, in case that $\tp(b/B)$ is a forking extension of $p$ we conclude
that $b\in{\rm acl}(\Pi_q',B)$ and so $\tp(b/B)$ has finite Lascar rank, as
desired.
\end{proof}

One can see directly that the structure $(\Z,+,0,\Pi_q)$ has infinite Lascar
rank, without using Theorem \ref{MainTheorem},
showing that the set $\Pi_q+\stackrel{n}{\ldots}+\Pi_q$ has Lascar rank $n$.
This is left to the reader.

\section{Generalizations}
In this section we would like to mention a few generalizations, concerning
proper superstable expansions of the integers, that follow from our methods.
The ideas that lie behind our proof are transparent and clear. Firstly one
reduces the superstability of the expanded structure to the superstability
of the induced structure on the new predicate.
Secondly
one needs to understand the induced structure in this new predicate. It seems
that this is equivalent to understanding its intersection with arithmetic
progressions
and with the solution set of linear equations over the integers.

The following example is not very different in nature with the ones we already
gave in the previous section, thus we leave its proof as an exercise to the
interested reader.

\begin{ex}\label{sp}
Let $(k_1,\ldots, k_m)$ be a sequence of natural numbers and
$${\rm SP}_{(k_1,\ldots,k_m)}:=\{k_1^{\iddots^{k_m^n}} \ | \ n<\omega\}.$$
Then
$(\Z,+,0,{\rm SP}_{(k_1,\ldots,k_m)})$ is
superstable of Lascar rank $\omega$
\end{ex}

A more interesting example is the subset of the integers consisting of factorial
elements, i.e. ${\rm Fac}:=\{ n! \ | \ n<\omega\}\cup\{0\}$.

\begin{proposition}\label{Fac}
The structure $(\Z,+,0,{\rm Fac})$ is superstable of Lascar rank $\omega$.
\end{proposition}

We first note that the set ${\rm Fac}$ satisfies
the following:

\begin{lemma}
A finite union of sets of the form
$$\{k_0+k_1\cdot\alpha_1+\ldots+ k_s\cdot\alpha_s \ | \
\alpha_1,\ldots,\alpha_s\in {\rm Fac}\},$$ where $k_0,\ldots,k_s$ are integers,
 cannot cover any coset of any (non-trivial) subgroup
of $\Z$.
 \end{lemma}
\begin{proof}
Suppose otherwise that the coset $m+n\Z$ is contained in such a finite union, and notice that we may assume that $m=0$.
By the Pigeonhole principle there are integers $\lambda_0,\ldots,\lambda_l$ determining one of these sets, a prime $p$ greater than $\lambda_0,\ldots,\lambda_l$
and an infinite subset $I_0$ of $\mathbb N$ such that $\{np^k\}_{k\in I_0}$ is contained in the set
$$
\left\{\lambda_0+\lambda_1\cdot\alpha_1+\ldots+\lambda_l\cdot\alpha_l \ | \ \alpha_1,\ldots,\alpha_l\in{\rm Fac} \right\}.
$$
Let $\alpha_1(k),\ldots,\alpha_l(k)$ denote the factorial numbers such that
$$
np^k=\lambda_0+\lambda_1\cdot\alpha_1(k)+\ldots+\lambda_l\cdot\alpha_l(k).
$$
Now, suppose that there is some infinite subset $I$ of $I_0$ such that for some $j$ the set $\{\alpha_j(k)\}_{k\in I}$ is finite.
Without loss of generality, we may assume $j=l$. Thus, by the Pigeonhole principle there is some factorial $\alpha$ and some infinite subset $I'$ of $I$ such that
$$
np^k=\lambda_0+\lambda_l\cdot\alpha+\lambda_1\cdot\alpha_1(k)+\ldots+\lambda_{l-1}\cdot\alpha_{l-1}(k),
$$
for $k$ in $I'$. Hence, after replacing $\lambda_0$ by $\lambda_0+\lambda_l\cdot\alpha$ and $I$ by a suitable infinite subset,
iterating this process, we may assume that for any infinite subset $I$ of $I_0$ the set $\{\alpha_j(k)\}_{k\in I}$ is unbounded for $1\le j\le l$.
Thus, we can find recursively on $j$ an infinite subset $I_j$ of $I_{j-1}$ such that $\alpha_1(k),\ldots,\alpha_j(k)$ are greater than $p!$ for every $k$ in $I_j$.
In particular, there is a natural number $k$, in $I_l$, for which the factorial numbers $\alpha_1(k),\ldots,\alpha_l(k)$ are greater than $p!$.
Consequently, as $p$ clearly divides $\lambda_1\cdot\alpha_1(k)+\ldots+\lambda_l\cdot\alpha_l(k)$, it also divides $\lambda_0$, a contradiction unless $\lambda_0=0$.
Therefore, we have shown that the set $\{np^k\}_{k\in I_0}$ is contained in
$$
\left\{\lambda_1\cdot\alpha_1+\ldots+\lambda_l\cdot\alpha_l \ | \ \alpha_1,\ldots,\alpha_l\in{\rm Fac} \right\}.
$$
However, this yields a contradiction since for arbitrarily large $k$ we can find a prime $q$ dividing $\lambda_1\cdot\alpha_1(k)+\ldots+\lambda_l\cdot\alpha_l(k)$
but not $np^k$. \end{proof}
Therefore, a similar proof as in Lemma \ref{Consistency} gives:

\begin{lemma}\label{1}
Let $\mathcal{L}$ be the language of groups and $P(x)$ be a unary predicate. Let
$\mathcal{Z}:=(\Z,+,0,{\rm Fac})$ be an $\mathcal{L}_P$-structure.
Then every $\mathcal{L}_P$-formula in $\mathcal{Z}$ is bounded.
\end{lemma}

We will next prove that the induced structure on ${\rm Fac}$ comes from
equality alone.

\begin{lemma}\label{2}
Let $k<n$ be natural numbers and let $[k]_n$ denote the congruence class of $k$
modulo $n$.
Then ${\rm Fac}\cap [k]_n$ is either finite or co-finite in ${\rm Fac}$.
\end{lemma}
\begin{proof}
It is easy to see that when $k$ is $0$ the intersection will be co-finite in
${\rm Fac}$, while in any other case the intersection will be finite.
\end{proof}

Given an equation $k_1x_1+\ldots+k_nx_n=0$ over the integers and a partition
$\mathcal P=\{I_j\}_{j\le l}$ of $\{1,\ldots,n\}$, we denote by $X_{\mathcal P}$
the set of solutions $(m_1!,\ldots,m_n!)$ such that $m_i=m_k$ if and only if
$i,k\in I_j$ for some $j\le l$.
\begin{lemma}\label{3}
Let $k_1x_1+\ldots+k_nx_n=0$ be an equation over the integers and let $\mathcal
P=\{I_j\}_{j\le l}$ be a partition of $\{1,\ldots,n\}$. Then the projection of
$X_{\mathcal P}$ on its $I_j$-coordinates is an infinite set if and only if
$\sum_{i\in I_j} k_i=0$.
\end{lemma}
\begin{proof}
Let $\mathcal P=\{I_j\}_{j\le l}$ be a partition of $\{1,\ldots,n\}$ and suppose
that $\sum_{i\in I_j} k_i=0$ for some $j\le l$. Clearly, there are infinitely
many solution of the form $(x_1,\ldots,x_n)$ with $x_i=0$ for $i\not\in I_j$ and
$x_i$ constant for $i\in I_j$. Hence, we get the result. For the converse,
assume for some $k\le l$ that the projection of $X_{\mathcal P}$ on its
$I_k$-coordinates yields an infinite set but $\sum_{i\in I_k} k_i$ is non-zero,
and let $X_{\mathcal P}$ be the set $\{(m_1(t)!,\ldots,m_n(t)!)\}_{t<\omega}$.
Set $s_j(t)$ to be the value of every $m_i(t)$ when $i\in I_j$, and note that
all $s_j(t)$'s are distinct by the definition of $X_{\mathcal P}$. It is clear
that
$$
\sum_{j\le l} \Big(\sum_{i\in I_j} k_i \Big) \cdot s_j(t)! =0.
$$
Now, let $J$ be the set of sub-indexes $j\le l$ for which $\sum_{i\in I_j} k_i$
is non-zero; note that $J$ is non-empty as $k\in J$ and also that
$$
\sum_{j\in J} \Big(\sum_{i\in I_j} k_i \Big) \cdot s_j(t)! =0.
$$
By assumption, this equation holds for all $t<\omega$ and so, by the pigeonhole
principle we can find an enumeration of $J=\{j_1,\ldots,j_r\}$ such that
$s_{j_1}(t)>\ldots>s_{j_r}(t)$ for infinitely many values of $t$. Additionally,
for some of these $t$'s we have that $s_{j_1}(t)>|\sum_{i\in I_{j_2}}
k_i+\ldots+\sum_{i\in I_{j_r}} k_i|$ and thus
$$
\left|\Big(\sum_{i\in I_{j_1}}k_i\Big)\cdot
s_{j_1}(t)!\right|>\left|\Big(\sum_{i\in I_{j_2}}k_i\Big)\cdot
s_{j_2}(t)!+\ldots+\Big(\sum_{i\in I_{j_r}}k_i\Big)\cdot s_{j_r}(t)!\right|,
$$
a contradiction. Hence, we get the result.
\end{proof}

If $k_1x_1+\ldots+k_nx_n=0$ is an equation over the integers and $S\subseteq
\Z^n$ is its solution set, observe that $S$ is precisely the finite union of all
$X_{\mathcal P}$. Therefore, Lemmata \ref{2} and \ref{3} give that all the
induced structure
on ${\rm Fac}$ comes from equality alone, thus ${\rm Fac}^{\rm ind}$ is
strongly minimal and
Proposition \ref{Fac} follows.

Our paper can be seen as opening the path for answering the following
interesting questions:

\begin{question}
\
\begin{itemize}
 \item (J. Goodrick) Characterize the subsets $\Pi\subset\Z$, for which $(\Z, +,
0, \Pi)$ is superstable.
 \item Characterize the subsets $\Pi\subset\Z$, for which $(\Z, +, 0, \Pi)$ is
stable.
\end{itemize}

\end{question}

\bibliographystyle{plain}

\end{document}